\documentclass{article}

\usepackage[utf8]{inputenc}
\usepackage{geometry}
\usepackage[english]{babel}

\usepackage{amsmath}
\usepackage{amsfonts}
\usepackage{amssymb}
\usepackage{amsthm}
\usepackage{mathrsfs}
\usepackage{tikz}
\usepackage{graphicx}
\usepackage{enumitem}
\usepackage{multicol}
\usepackage{verbatim}
\usepackage{multirow}
\usepackage{array}
\usepackage{tabularx}
\usepackage{colortbl}
\usepackage{caption}
\usepackage{subcaption}
\usepackage{soul}
\usepackage{mathtools}
\usepackage{dsfont}

\newtheorem{theorem}{Theorem}

\newtheorem{lemma}[theorem]{Lemma}
\newtheorem{prop}[theorem]{Proposition}
\newtheorem{conjecture}[theorem]{Conjecture}
\newtheorem{claim}[theorem]{Claim}

\theoremstyle{definition}

\theoremstyle{remark}
\newtheorem*{remark}{Remark}

\newcommand{\N}{\mathbb{N}}
\newcommand{\Z}{\mathbb{Z}}

\newcommand{\C}{\mathbb{C}}
\newcommand{\F}{\mathbb{F}}

\def\P{{\hbox{\bf P}}}

\renewcommand{\mod}{\text{ (mod }}

\title{Avoiding Monochromatic Solutions 
to 3-term Equations}

\author{
Kevin P. Costello \\ kevin.costello@ucr.edu
\and 
Gabriel Elvin \\ gelvin@math.ucr.edu}

\date{University of California, Riverside \\[2ex] \today}

\begin{document}

\maketitle

\begin{abstract}
Given an equation, the integers $[n] = \{1, 2, \dots, n\}$
as inputs,
and the colors red and blue,
how can we color $[n]$  in order to minimize the 
number of monochromatic solutions to the equation,
and what is the minimum?
The answer is only known for a handful of equations,
but much progress has been made on improving upper
and lower bounds on minima for various equations. 
A well-studied characteristic an equation,
which has its roots in graph Ramsey theory,
is to determine if the minimum 
number of monochromatic solutions can be achieved
(asymptotically) by uniformly random colorings. 
Such equations are called \textit{common}.
We
prove that no 3-term equations are common and
provide a lower bound for a specific 
class of 3-term equations.
\end{abstract}

\section{Introduction}

Given an equation
\begin{equation}
a_1 x_1 + \cdots + a_k x_k = 0 \text{ with } a_i \in \Z
\end{equation}
and the colors red and blue,
how should we color the elements of
$[n] = \{1, 2, \dots, n\}$
in order to reduce the number 
of monochromatic solutions,
with the ultimate goal being to find the 
asymptotic (as $n \to \infty$) minimum number?
To be precise, by a coloring we mean 
a function $f : [n] \to \{-1, 1\}$
(where $-1$ represents blue and $1$ represents red),
by a solution we mean a vector $(x_1^*,\dots,x_k^*) \in [n]^k$
that satisfies the equation,
and by monochromatic we mean $f(x_1^*) = \cdots = f(x_k^*)$.

Before we proceed, we clarify the asymptotic 
notation used throughout.
Let $f, g$ be functions of $n$.
If $f = O(g)$, there exist some constants $C, N$
such that $|f(n)| \leq C g(n)$ for all $n \geq N$.
By $f = \Omega(g)$, we mean $g = O(f)$.
If $f = o(g)$, this indicates $|f|/ g \to 0$ as 
$n \to \infty$\footnote{In all cases, the implicit
constants are allowed to depend on the equation
being analyzed.}.

Asymptotic minima are difficult to come by,
but much progress has been made on improving 
upper and lower bounds.
The most comprehensive result on lower bounds
is due to Frankl, Graham, and R\"odl,
who showed that as long there is a nonempty subset
of coefficients which sum to 0,
the equation will always have $\Omega(n^{k-1})$ monochromatic solutions  \cite{FGR}.
In fact, their result is more general, considering systems of equations
and an arbitrary number of colors.

For upper bounds, a well-studied problem is to determine
if colorings can be found which yield fewer monochromatic
solutions asymptotically than uniformly random colorings.
This problem has its roots in graph Ramsey theory,
where one can ask a similar question:
given a fixed graph $H$,
can the edges of $K_n$ always be colored
in such a way that produces asymptotically fewer
monochromatic copies of $H$ in $K_n$ than
what would be expected from uniformly random colorings?
Graphs with this property are referred to 
as \textit{uncommon}.
In 1959, Goodman showed that $K_3$ was \textit{common},
i.e. every coloring of $K_n$ has asymptotically
at least as many monochromatic copies of $K_3$
as one would expect from uniformly random colorings
\cite{G}.
Three years later, Erd\H{o}s conjectured 
that $K_s$ was common for all $s \geq 2$ \cite{E},
and in 1980 Burr and Rosta were even bolder,
conjecturing that all graphs were common \cite{BR}.
However, in 1989 Thomason showed that
$K_4$ was uncommon, disproving both conjectures 
\cite{Thomason}.

The first result regarding equations
came nearly a decade later,
and we will highlight certain aspects
of the original equation of study:
$x + y = z$, known as \textit{Schur's equation}.
Each solution is generally represented as a \textit{Schur triple}
$(x, y, x + y)$. 
There are $\binom{n}{2}$
solutions over $[n]$
(when $(x,y,x + y)$ and $(y,x,x + y)$
are considered distinct).
With a uniformly random coloring, 
a given solution will be monochromatic
with probability 1/4, so we would expect
$n^2/8 + O(n)$ monochromatic solutions.
In 1998, 
Robertson and Zeilberger 
found the asymptotic minimum 
number of monochromatic solutions:
$n^2 / 11 + O(n)$ \cite{RZ}.
In particular,
there is always a coloring of $[n]$
with fewer
monochromatic solutions than what would be
expected from a uniformly random coloring.
To borrow the terminology from graph theory,
$x + y = z$ is \textit{uncommon}.
A coloring
that achieves this minimum is 
quite simple to describe:
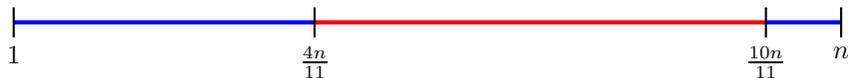
\begin{center}
\begin{tikzpicture}
\draw[ultra thick, blue] (0,0) -- (4,0);
\draw[ultra thick, red] (4,0) -- (10, 0);
\draw[ultra thick, blue] (10,0) -- (11,0);
\draw[thick] (0,0.2) -- (0,-0.2);
\draw (0,-0.2) node[anchor = north]{1};
\draw[thick] (4,0.2) -- (4,-0.2);
\draw (4,-0.2) node[anchor = north]{$\frac{4n}{11}$};
\draw[thick] (10,0.2) -- (10,-0.2);
\draw (10,-0.2) node[anchor = north]{$\frac{10n}{11}$};
\draw[thick] (11,0.2) -- (11,-0.2);
\draw (11,-0.2) node[anchor = north]{$n$};
\end{tikzpicture}
\end{center}
(or as close to this as possible when $n$ is not a multiple of $11$).

One can ask the same questions about other equations
or systems of equations.
Generally the equations considered are 
linear with integer coefficients,
which enables another variation on the problem:
replace $[n]$ with an abelian group.
Colorings of $\Z_n$ and $\F_p^n$ are frequently studied
\cite{LP, FPZ, SW}.
Over $[n]$, true (asymptotic) minima are only known for 
a handful of equations,
such as $x + by = z$ with $b \in \N = \{1,2,3,\dots\}$
\cite{RZ, TW} and $x + y = z + w$,
where it turns out the $1/8$ fraction of monochromatic solutions 
from a random coloring is asymptotically optimal (see Appendix \ref{xyzw-app}
for proof).
Asymptotic minima are studied most often,
but there are also some results on \textit{exact}
minima \cite{KW}.

In this paper, we restrict our focus to 3-term equations
and address both upper and lower bounds.
For upper bounds, we show that all 3-term equations
are \textit{uncommon}, i.e. 
we can always color $[n]$
in such a way which produces 
asymptotically fewer monochromatic solutions
than what is expected from uniformly random colorings.
This result is of interest because 
all 3-term equations 
(in fact, all equations with an odd number of terms)
are actually \textit{common}
over any abelian group whose order is coprime
to each coefficient of the equation \cite{CCS}.
For lower bounds, we use a structure theorem
(a robust version of Freiman's $3k - 4$ Theorem \cite{SX})
to show equations of the form $ax + ay = cz$, $a,c \in \N$,
always have $\Omega(n^2)$ monochromatic solutions.

Even through the main focus in this paper is 3-term equations,
the notation below is kept more general
in order to make connections with the result in \cite{V},
which we will discuss later.
Let $E$ be an equation $a_1 x_1 + \cdots + a_k x_k = 0$ with integer
coefficients whose inputs are taken from a finite set $A$,
and let $f : A \to \{-1, 1\}$ a coloring.
Denote the set of all solutions
\begin{equation}
T_E(A)
\coloneqq \{(x_1,\dots,x_k) \in A^k \,|\, a_1 x_1 + \cdots + a_k x_k = 0\},
\end{equation}
and denote the set of all monochromatic solutions
\begin{equation}
M_E(f) \coloneqq
\{(x_1,\dots,x_k) \in T_E(A) 
	\,|\, f(x_1) = \cdots = f(x_k)\}.
\end{equation}
Next, the proportion of monochromatic solutions
under $f$ is denoted
\begin{equation}
\mu_E(f) \coloneqq \frac{|M_E(f)|}{|T_E(A)|}.
\end{equation}
Finally, the value in question is the minimum 
monochromatic proportion:
\begin{equation}\label{proportion-mc-eq}
\mu_E(A) \coloneqq \min_{f : A \to \{\pm 1\}}
\mu_E(f).
\end{equation}

\section{Upper Bounds}\label{upper_bounds}

With the previous notation,
an equation $E$ is \textbf{uncommon over $[n]$} if
\begin{equation}\label{eq_def_uncommon}
\limsup_{n \to \infty} \mu_E([n]) = \frac{1}{2^{k-1}} - \Omega(1)
\end{equation}
(asymptotically strictly less than $2^{1-k}$).
To reiterate, $2^{1-k}$ is the expected value of 
$\mu_E(f)$ when $f$ is a uniformly random coloring.
Note that what makes an equation uncommon over $[n]$ 
is a \textit{sequence} of colorings in $n$, 
but will often refer to the sequence simply
as a single coloring $f : [n] \to \{-1,1\}$
defined in terms of $n$.
With this, we can now state our main result formally.
\begin{theorem}\label{thm_main_result}
All equations $ax + by + cz = 0$ with $a,b,c \in \Z$
are uncommon over $[n]$.
\end{theorem}
Here we emphasize ``over $[n]$''
because of other results when $[n]$
is replaced by an abelian group
\cite{CCS, LP, SW, FPZ}.
To show many equations are uncommon over $[n]$,
we will color cyclic groups and extend them to 
colorings of $[n]$,
an idea that has some similarities
with techniques for solving 
related problems \cite{LP, BGL}.
We will actually define our colorings
via probability distributions and use 
Fourier-analytic techniques like those
in \cite{CCS, D, FPZ, SW, V}.
We do not make a distinction between the general cyclic group
of order $m$ and the integers modulo $m$, which 
we will denote $\Z_m = \{0, 1, \dots, m-1\}$.

Below is the crucial lemma that allows us to work
in $\Z_m$ rather than $[n]$. 
An analogous statement can be found in \cite{LP}
regarding arithmetic progressions.
While our results are centered around 3-term equations,
we state this fact more generally for use in a later discussion.
\begin{lemma}\label{lem_mod_ok}
Given an equation $E : a_1 x_1 + \cdots + a_k x_k = 0$
and a positive integer $m$,
\begin{equation}\label{mod-ok-eq}
\limsup_{n\to\infty}\mu_E([n]) \leq \mu_E(\Z_m).
\end{equation}
\end{lemma}

\begin{proof}
Let $f : \Z_m \to \{-1, 1\}$ be a coloring that achieves
the minimum on the right-hand side. This coloring can be extended
to a coloring $\tilde{f} : [n] \to \{-1, 1\}$ very naturally
by composing $f$ with the canonical projection map $[n] \to \Z_m$.
By design, a vector in $[n]^k$ is monochromatic if and only if
it is monochromatic when projected onto the corresponding vector in $\Z_m^k$. 
Let $Cn^{k-1} + O(n^{k-2})$ be the number of solutions to the equation over $[n]$,
where $C$ is some positive constant. 
Then each solution over $\Z_m$ corresponds to 
\begin{equation*}
C\left(\frac{n}{m}\right)^{k-1} + O(n^{k-2})
\end{equation*}
solutions over $[n]$.
Using the fact that $|T_E(\Z_m)| = m^{k-1}$,
we get
\begin{equation*}
\mu_E([n]) 
\leq \frac{|M_E(\tilde{f})|}{|T_E([n])|}
= \frac{\mu_E(\Z_m) m^{k-1}[C(n/m)^{k-1} + O(n^{k-2})]}{Cn^{k-1} + O(n^{k-2})}
= \mu_E(\Z_m) + o(1),
\end{equation*}
and the result follows.
\end{proof}
Lemma \ref{lem_mod_ok} is critical because it allows us to prove
results (and use past results) over $\Z_m$ and apply them to scenarios
over $[n]$.
In practice, the ``colorings'' we use are actually defined probabilistically,
and we invoke the probabilistic method to say that if there is a random
coloring whose expected proportion of monochromatic solutions is at most some value $K$,
then there must exist an actual coloring $f$ such that $\mu_E(f) \leq K$.

\begin{remark}
In the graph theoretic setting, the proportion
analogous to $\mu$,
\[\frac{\text{min. \# of monochr. $H$ in $K_n$}}{\text{total \# of $H$ in $K_n$}},\]
has a limit as $n \to \infty$ (often referred to as the
\textit{Ramsey multiplicity constant}).
The proof of this fact is straightforward,
as the sequence is bounded and monotonic.
However, for equations the corresponding sequence
is not monotonic.
We still expect the limit to exist, but a proof (or counterexample)
has not yet been found.
\end{remark}

We prove Theorem \ref{thm_main_result} in a series of steps,
each of which handles some subset of 3-term equations 
$ax + by + cz = 0$, $a, b, c \in \Z$.
First, we use
Fourier-analytic techniques and Lemma \ref{lem_mod_ok}
to deal with most equations. 
Next, we state and prove a modest proposition for nearly all the equations
not covered in the first step and again utilize Lemma \ref{lem_mod_ok}.
Finally, the equations which remain are a small and rigid class of equations
and one isolated equation, and for these we explicitly 
define colorings with asymptotically fewer than the critical
$1/4$ monochromatic fraction expected from uniformly random colorings.
Detailed computations for the equations in this step are provided in
Appendix \ref{app_computations}.
We always assume the equations are fully reduced,
i.e. $\gcd(a, b, c) = 1$.

\subsection{Fourier-analytic Techniques}\label{subsec_fourier}

First, we will cover the standard notation for Fourier analysis
in this setting. For a generalized and thorough introduction, 
we recommend \cite[Chapter 4]{TV}.
Let $f : \Z_m \to [0, 1]$, which we associate with a probabilistic coloring via
\begin{equation}
f(t) = \P[t \text{ is red}].
\end{equation}
When we identify elements in $\Z_m$ with the integers
$0, 1, \dots, m - 1$, both
\begin{equation*}
f(t) \quad \text{and} \quad e^{-2\pi i \xi t / m} \quad (t, \xi \in \Z_m)
\end{equation*}
well-defined notions. 
The \textbf{Fourier transform of} $f$, denoted $\widehat{f}$,
is the function from $\Z_m$ to $\C$ given by
\begin{equation}
\widehat{f}(\xi) \coloneqq \frac{1}{m}\sum_{t \in \Z_m} f(t) e^{-2\pi i \xi t/m}.
\end{equation}
In the arguments of \cite{CCS, FPZ, V}, it was crucial that the order
of the group was relatively prime to each coefficient of the equation. 
We will use similar tools, but we will actually exploit the
fact that these results do not always hold without this condition.

Without loss of generality we may assume $|c| = \max\{|a|, |b|, |c|\}$. Put $m = |c|$.
In order to use Fourier transforms effectively, we need 
additional assumptions:
\begin{align}
& m > |a|, |b|, \label{eq_unique_max}\\
& \text{One of } \gcd(a, m), \; \gcd(b, m) \text{ is equal to } 1, \label{eq_low_gcd}\\
& a + b  \not\equiv 0 \mod m). \label{eq_zero_sum}
\end{align}
Our goal now is to show that all equations of this form are uncommon
over $\Z_m = \Z_{|c|}$, as this combined with Lemma \ref{lem_mod_ok}
implies they are also uncommon over $[n]$.
We will then show, using various other techniques,
that the equations not satisfying
one of the above assumptions are still uncommon.

First, 
we can write the expected proportion of red solutions
in terms of Fourier transforms: 
\begin{equation}
\widehat{f}(0)\sum_{t \in \Z_m}\widehat{f}(a t) \widehat{f}(b t).
\end{equation}
Note, this formula requires (\ref{eq_unique_max}).
Extending this idea, the expected number of monochromatic solutions is
\begin{equation}\label{prop_mc_fourier}
\mu_{\{ax+by+cz=0\}}(f) 
= \widehat{f}(0)\sum_{t \in \Z_m}\widehat{f}(a t) \widehat{f}(b t)
+ \widehat{(1-f)}(0)\sum_{t \in \Z_m}\widehat{(1-f)}(a t) \widehat{(1-f)}(b t).
\end{equation}
Therefore, to show $ax + by + cz = 0$ is uncommon over $\Z_m$,
we simply need to find an $f$ such that
(\ref{prop_mc_fourier})
is strictly less than $1/4$.
In order to simplify calculations, we will impose the restriction
$\widehat{f}(0) = 1/2$, which is equivalent to requiring overall red and blue
appear with equal probability.
This gives us
\begin{align*}
\mu_{\{ax+by+cz=0\}}(f)
&= \frac{1}{4} + \frac{1}{2}\sum_{t \in \Z_m - \{0\}}
[\widehat{f}(a t) \widehat{f}(b t) + \widehat{(1-f)}(a t) \widehat{(1-f)}(b t)] \\
&= \frac{1}{4} + \sum_{\substack{t \in \Z_m \\ at, bt \neq 0}}\widehat{f}(at)\widehat{f}(bt).
\end{align*}
The last equality follows from the fact that 
$\widehat{(1-f)}(s) = -\widehat{f}(s)$
whenever $s \neq 0$, so any summand in the first sum 
with exactly one of $at$, $bt$ equal to 0 
will be 0 (and the case $at = bt = 0$ will only
occur when $t = 0$ since the equation is fully reduced).
Therefore, it suffices to find an $f$ such that
$\widehat{f}(0) = 1/2$ and
\begin{equation}\label{deviation}
\sum_{\substack{t \in \Z_m \\ at, bt \neq 0}}\widehat{f}(at)\widehat{f}(bt) < 0.
\end{equation}
We will refer to the above sum as the \textbf{deviation}.
By the Fourier inversion formula, we may define
$f$ by its Fourier coefficients, although some care must
be taken to ensure $\text{Range}(f) \subseteq [0,1]$.
First, we will utilize the fact that
$f$ is real-valued if and only if $\widehat{f}$ is Hermitian:
$\overline{\widehat{f}(s)} = \widehat{f}(-s)$ for all $s$.
Next, we must find Fourier coefficients that guarantee 
$f$ is between $0$ and $1$.
To do this, we will use the \textbf{Fourier inversion formula}:
\begin{equation}\label{eq_fourier_inversion}
f(v) = \sum_{t \in \Z_m} \widehat{f}(t) e^{2\pi i t v / m}.
\end{equation}
By requiring $\widehat{f}(0) = 1/2$ and using the triangle
inequality with (\ref{eq_fourier_inversion}), we have
\begin{equation}\label{eq_fourier_triangle}
|f(v) - 1/2| \leq \sum_{t \in \Z_m - \{0\}}|\widehat{f}(t)|.
\end{equation}

Regarding Assumption (\ref{eq_low_gcd}),
without loss of generality we may assume $\gcd(a, m) = 1$.
We split the work into two cases:  $a \neq b$ and $a = b$.
When $a \neq b$, we may set $\widehat{f}(\pm a) = -1/8$, 
$\widehat{f}(\pm b) = 1/9$, and $\widehat{f}(s) = 0$
for all other $s \neq 0$.
Note that if we did require (\ref{eq_zero_sum}),
these choices could not be made.
With this $\widehat{f}$ is Hermitian, and
by (\ref{eq_fourier_triangle}) $0 \leq f(v) \leq 1$
for all $v$.
Now we argue the deviation is negative.
Here, the deviation will have at least two negative terms
and at most two positive terms. 
To see this, the negative terms are guaranteed by $t = \pm 1$,
which are distinct since (\ref{eq_unique_max}) and
(\ref{eq_zero_sum}) together imply $m \geq 3$.
Positive terms arise when 
\[(at, bt) \in \{(\pm a, \pm a), (\pm a, \mp a), (\pm b, \pm b), (\pm b, \mp b)\}.\]
Since $\gcd(a, m) = 1$, $at$ has a unique solution modulo $m$,
so $(at, bt) \in \{(\pm a, \pm a), (\pm a, \mp a)\}$ will only occur
when $t = \pm 1$. In the other cases, $t = \pm b a^{-1} \notin \{\pm 1\}$
is possible.
Therefore, the deviation is at most
\[-2\cdot \frac{1}{8}\cdot\frac{1}{9} + 2\cdot \frac{1}{9^2} < 0,\]
and hence the equation is uncommon over $\Z_m$.

If $a = b$, then we simply take $\widehat{f}(\pm a) = \pm i/4$
and $\widehat{f}(s) = 0$ for all other $s \neq 0$.
Again, $\widehat{f}$ is Hermitian and $f$ takes
values within $[0, 1]$, and here
the deviation $-1/8$.
This covers all cases, proving any equation satisfying the initial
assumptions (\ref{eq_unique_max}), (\ref{eq_low_gcd}), and (\ref{eq_zero_sum})
is uncommon over $\Z_m$ and is therefore uncommon over $[n]$ by Lemma \ref{lem_mod_ok}.
Next we will cover equations that do not satisfy those assumptions.

\subsection{Remaining Equations}

As discussed previously, the Fourier-analytic techniques
do not cover every equation. 
Recall the assumptions we needed:
\begin{align*}
(\ref{eq_unique_max}) \quad & m > |a|, |b|, \\
(\ref{eq_low_gcd}) \quad & \text{One of } \gcd(a, m), \; \gcd(b, m) \text{ is equal to } 1, \\
(\ref{eq_zero_sum}) \quad & a + b \not\equiv 0 \mod m).
\end{align*}
If (\ref{eq_low_gcd}) does not hold, then we may assume one of
these gcds is at least 3, as they cannot both be 2 with the equation
fully reduced.
For these equations, we have the following proposition.

\begin{prop}\label{prop_high_gcd}
Every 3-term equation with two coefficients that have 
a common factor of at least 3 is uncommon over $[n]$.
\end{prop}

\begin{proof}
Without loss of generality, assume $m = \gcd(a, c) \geq 3$.
As done previously, we will work in $\Z_m$.
The coloring is quite simple: $f(0) = -1$, and $f(t) = 1$ otherwise.
Note that since the equation is fully reduced, 
every solution will be of the form $(x, 0, z) \in \Z_m^3$,
and $x,z$ are unrestricted.
Therefore, only one solution, namely $(0, 0, 0)$, will be monochromatic,
and hence the monochromatic proportion is $1/m^2 \leq 1/9 < 1/4$.
By Lemma \ref{lem_mod_ok} this extends to a coloring of $[n]$,
and hence the equation is uncommon over $[n]$.
\end{proof}

Equations where (\ref{eq_zero_sum}) does not hold are equivalent to one
of two types of equations:
\[ax + by = (a + b)z \quad \text{and} \quad
ax - ay + cz = 0 \quad (a, b, c \in \N).\]
The first type is equivalent to a \textit{constellation}
shown to be uncommon in \cite{BCG}.
Equations of the second type with $|a| \geq 3$
can be eliminated by Proposition \ref{prop_high_gcd},
which does not require that $c$ is the largest
coefficient.
If $|a| = 1$, the equations are equivalent to ones of the form $x - y + cz = 0$,
which were shown to be uncommon in \cite{TW} (in fact, the authors
found asymptotic minima).
If $|a| = 2$, we are left with equations of the form
\begin{equation}
2x - 2y + cz = 0.
\end{equation}

If (\ref{eq_unique_max}) does not hold but the largest
coefficients are at least 3, 
then Proposition \ref{prop_high_gcd}
ensures these equations are uncommon.
Up to equivalence, the equations left in this case are
\begin{equation*}
x + y - z = 0, \quad 2x - y + 2z = 0, \quad \text{and} \quad
2x + y - 2z = 0.
\end{equation*}
The first equation is Schur's equation, discussed previously.
Therefore, the only equations not yet covered are, up to equivalence:
\begin{equation}
2x - 2y + cz = 0 \;\; \qquad \text{and} \qquad 2x - y + 2z = 0.
\end{equation}
We now describe colorings for these equations
that yield asymptotically fewer than a $1/4$ proportion
of monochromatic solutions, 
and detailed computations can be found in Appendix B.

The colorings for equations of the form $2x - 2y + cz = 0$
all have a similar construction:
alternate between red and blue until some boundary point $\alpha n$
that depends on $c$,
and then color from $\alpha n$ to $n$ entirely red.
Let the coloring $f : [n] \to \{-1, 1\}$ be defined as follows:
\[f(t) = \begin{cases}
-1, & t \text{ even, } t \leq \alpha n, \\
1, & \text{otherwise},
\end{cases}
\qquad \text{where} \qquad
\alpha = \begin{cases}
3/4, & c = 1, \\
2/c, & c \geq 3
\end{cases}\]
(note that $c$ is odd because our equations
are fully reduced).
With these colorings, we get monochromatic proportions of
\[\begin{cases}
5/24 + o(1), & c = 1, \\
1/c^2 + o(1), & c \geq 3,
\end{cases}\]
both of which are asymptotically less than $1/4$,
proving these equations are uncommon.

For the final equation, $2x - y + 2z = 0$,
we use the following coloring:
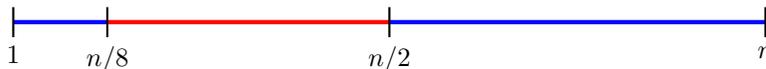
\begin{center}
\begin{tikzpicture}
\draw[ultra thick, blue] (0,0) -- (1.25,0);
\draw[ultra thick, red] (1.25,0) -- (5, 0);
\draw[ultra thick, blue] (5,0) -- (10,0);
\draw[thick] (0,0.2) -- (0,-0.2);
\draw (0,-0.2) node[anchor = north]{1};
\draw[thick] (1.25,0.2) -- (1.25,-0.2);
\draw (1.25,-0.2) node[anchor = north]{$n/8$};
\draw[thick] (5,0.2) -- (5,-0.2);
\draw (5,-0.2) node[anchor = north]{$n/2$};
\draw[thick] (10,0.2) -- (10,-0.2);
\draw (10,-0.2) node[anchor = north]{$n$};
\end{tikzpicture}
\end{center}
(or as close to this as possible if $n$ is not a multiple of $8$).
With this coloring, the proportion of monochromatic solutions is
$1/64 + o(1)$, far less than the $1/4$ threshold.
This finally proves Theorem \ref{thm_main_result}, i.e.
all 3-term equations are uncommon over $[n]$.
Next, we will calculate lower bounds for a specific class
of 3-term equations.

\section{Lower Bounds}

Every equation
\begin{equation}
a_1 x_1 + \cdots + a_k x_k = 0 \;\; (a_i \in \Z)
\end{equation}
has $C n^{k-1} + O(n^{k-2})$ solutions,
and we believe a positive fraction of these
will always be monochromatic.
We state this another way with the following 
conjecture.
\begin{conjecture}
Given an equation $a_1 x_1 + \cdots + a_k x_k = 0$,
$a_i \in \Z$,
every coloring has $\Omega(n^{k-1})$ monochromatic solutions
over $[n]$.
\end{conjecture}
As stated previously, a result of Frankl, Graham, and R\"{o}dl
confirms this conjecture
for equations which have a subset of coefficients that sum to 0 \cite{FGR}.
And in fact, they showed this for systems of equations (with an 
analogous assumption on the coefficients) and colorings of
an arbitrary number of colors.
They also showed that this is not necessarily true
for equations in general via the equation
$x + y - 3z = 0$ using 5 colors.
We expect this lower bound on the number of solutions
to still hold when only two colors are used.
We make partial progress towards this conjecture .
\begin{theorem}\label{thm_lower_bound}
Equations of the form
$ax + ay - cz = 0$ ($a, c \in \N$)
always have $\Omega(n^2)$ monochromatic solutions.
\end{theorem}
We will prove this by using the structure theorem from
Xuancheng Shao and Max Wenqiang \cite{SX}.
We may assume $a$ and $c$ are relatively prime.
Fix a coloring $f : [n] \to \{-1, 1\}$, and let 
$R = f^{-1}(\{1\})$ and $B = f^{-1}(\{-1\})$
denote the red and blue elements, respectively.
We will actually show there are $\Omega(n^2)$ monochromatic
solutions just among the multiples of $c$, so we denote
$R' = R \cap c\Z$ and $B' = B \cap c\Z$.
Let $C_1$ and $C_2$ be small, positive constants 
possibly depending on $a$ and $c$ to 
be determined later.

\begin{claim}\label{claim_large_sets}
If $|R'| \leq C_1 n$ for sufficiently small $C_1$, then there
are $\Omega(n^2)$ blue solutions. 
\end{claim}

\begin{proof}
There are $\Omega(n^2)$ total solutions involving only 
multiples of $c$. 
Since each number in $R'$ is present in at most $3n$ solutions,
by assumption there are at most $3C_1 n^2$ solutions with 
an input from $R'$.
If we make $C_1$ small enough,
this still leaves $\Omega(n^2)$ solutions with inputs 
exclusively from $B'$.
\end{proof}

By this claim, we may assume $|R'|, |B'| \geq C_1 n$.
Now we will cover some necessary notation.
Let $X, Y \subseteq \Z$.
The \textbf{sum set} of two sets $X$ and $Y$,
denoted $X + Y$, is
\begin{equation}
X + Y = \{x + y \,:\, x \in X, y \in Y\}.
\end{equation}
The basic outline of our argument is as follows:
if the sum sets $R' + R'$ and $B' + B'$ are both large, 
then they will
have a nontrivial intersection, 
and if one of these sum sets is small, 
then $[n] \cap c\Z$ will contain large 
monochromatic arithmetic progressions,
and both cases imply there will be $\Omega(n^2)$
monochromatic solutions.
Rather than use sum sets,
We will use the \textit{robust} sum sets
defined in \cite{SX}:
given a subset $\Gamma \subseteq X \times Y$,
let
\begin{equation}
X +_\Gamma Y \coloneqq \{x + y \,:\, (x,y) \in \Gamma\}.
\end{equation}
For $A \in \{R', B'\}$,
let $\Gamma = \Gamma(A)$ be the set
of all pairs in $A \times A$ whose sum
has at least $C_2 n$ distinct 
representations as a sum of pairs, i.e.
\begin{equation}\label{eq_gamma}
\Gamma = \{(a_1, a_2) \in A \times A \,:\,
|\{(b_1, b_2) \in A \times A \,:\,
b_1 + b_2 = a_1 + a_2\}| \geq C_2 n\}.
\end{equation}

We will first show that if 
the robust sum sets in question are large,
then we will have $\Omega(n^2)$ monochromatic solutions.
Let $\epsilon > 0$. We leave it arbitrary for now,
but later we will pick a specific $\epsilon$
which depends on $C_1$ and $C_2$.

\begin{claim}
If $|A +_{\Gamma(A)} A| \geq (2 + \epsilon)|A|$
for $A = R',B'$, then
\begin{equation}\label{eq_big_intersection}
|(R' +_{\Gamma(R')} R') \cap (B' +_{\Gamma(B')} B')| = \Omega(n),
\end{equation}
which implies there are $\Omega(n^2)$ monochromatic solutions.
\end{claim}

\begin{proof}
We have
\[|R' +_{\Gamma(R')} R'| + |B' +_{\Gamma(B')} B'| 
\geq (2 + \epsilon)(|R'| + |B'|)
= (2 + \epsilon)\left\lfloor\frac{n}{c}\right\rfloor,\]
and since $A +_{\Gamma(A)} A \subseteq c\Z \cap [2n]$
(which has only $\lfloor 2n/c \rfloor$ elements),
(\ref{eq_big_intersection}) follows
from the Inclusion-Exclusion Principle.

Note that by construction
$v \in (R' +_{\Gamma(R')} R') \cap (B' +_{\Gamma(R')} B')$
corresponds to at least $C_2 n$
monochromatic solutions: if $v$ is colored red,
each distinct representation will correspond 
to a red solution, and similarly if $v$ is colored blue.
Since there are $\Omega(n)$ such $v$,
we have $\Omega(n^2)$ monochromatic solutions.
\end{proof}

Because of the above claim, 
we may now assume that one of the robust sum sets
is not too large. Without loss of generality,
suppose
\begin{equation}\label{eq_small_intersection}
|R' +_{\Gamma(R')} R'| < (2 + \epsilon)|R'|.
\end{equation}

We are now in a position to use the previously
mentioned structure theorem \cite{SX}.
Rather than state the theorem verbatim,
we state only what we need for this scenario.

\begin{theorem}\label{thm_sx}
Let $\epsilon > 0$. Suppose $|R'| \geq \max\{3, 2\epsilon^{-1/2}\}$,
and let $\Gamma \subseteq R' \times R'$ be a subset with
$|\Gamma| \geq (1 - \epsilon)|R'|^2$. If $|R' +_\Gamma R'|
< (1 + \theta - 11\epsilon^{1/2})|R'|$, where $\theta = \frac{1 + \sqrt{5}}{2}$,
then there is an arithmetic progression $P$ with
$|P| \leq |R' +_\Gamma R'| - (1 - 5\epsilon^{1/2})|R'|$,
$|R' \cap P| \geq (1 - \epsilon^{1/2})|R'|$.
\end{theorem}

With (\ref{eq_small_intersection}), $R' +_\Gamma R'$
is small enough to fit the corresponding assumption to the theorem.
We also have an appropriate lower bound on $\Gamma$.
To see this, note the following claim.

\begin{claim}
There are at most $\epsilon |R'|^2$ pairs in $R' \times R' - \Gamma$,
where $\epsilon = \frac{2C_2}{c C_1^2}$.
\end{claim}

\begin{proof}
Since $R' + R'$ contains only multiples of $c$ and lies inside $[2n]$,
$|R' + R'| \leq 2n/c$. By the definition of $\Gamma$,
each of these elements leads to at most $C_2 n$ pairs 
that are not in $\Gamma$.
Therefore, since $|R'| \geq C_1 n$,
\[|R' \times R' - \Gamma| \leq (C_2 n)(2n/c) \leq \frac{2C_2}{c C_1^2}|R'|^2.\]
\end{proof}

By this claim,
\[|\Gamma| = |R' \times R' - \Gamma| \geq |R'|^2 - \epsilon |R'|^2
= (1 - \epsilon)|R'|^2,\]
as required. 

Note that the line of reasoning
from Claim \ref{claim_large_sets} up to this point
is valid for $B'$ as well
(with the same choice for $\epsilon$), 
and that once $C_1$ is fixed (by Claim \ref{claim_large_sets})
this choice of $\epsilon$ can be made arbitrarily small by
decreasing $C_2$.

By Theorem \ref{thm_sx}, we can now say that
$R'$ strongly resembles an
arithmetic progression.
To be precise, we must first introduce a bit of notation:
if $f = o_{\epsilon}(1)$, then 
$f \to 0$ as $\epsilon \to 0$.
Now the strong resemblance $R'$ has to $P$ means
\begin{equation}
|R' \cap P| = (1 - o_\epsilon(1))|R'|.
\end{equation}
With so much information about the coloring
(at least on the multiples of $c$),
we can now find a specific progression 
which contains the desired amount of monochromatic solutions.

\begin{claim}
There exists an arithmetic progression $Q = \{dk \,:\, 1 \leq k \leq \lfloor n/d \rfloor\}$ 
with $|Q \cap A| = (1 - o_\epsilon(1))|Q|$
for some $A \in \{R, B\}$.
\end{claim}

\begin{proof}
If $P = \{a + dk\}$ has these properties, 
then we are done, so suppose it does not.
We may assume $0 < a < d$.
Then the progression $Q = \{dk\}$ is 
contained almost entirely in $B'$
(since it's almost entirely disjoint
from $P$ and everything is still a multiple
of $c$).
To be precise, 
$|Q \cap P| = o_\epsilon(1)$,
so $|Q \cap B'| = (1 - o_\epsilon(1))|Q|$.
\end{proof}

Finally, if $x,y \in Q$, then
\begin{equation*}
z = \frac{x + y}{c} = \frac{dk_1 + dk_2}{c}
= d\left(\frac{k_1 + k_2}{c}\right),
\end{equation*}
which is in $Q$ whenever $k_1 + k_2 \in c\Z$,
with at most a constant number of exceptions.
This will happen about $1/c$ of the time, so
\begin{equation*}
\text{\# of monochr. solns. in $Q$}
= \frac{1}{c}(1 - o_\epsilon(1)|Q|)^2 
= \frac{1}{c}(1 - o_\epsilon(1))\left(\frac{n}{d}\right)^2
= \Omega(n^2).
\end{equation*}
Therefore, every equation of the form $ax + ay - cz = 0$
has $\Omega(n^2)$ monochromatic solutions
regardless of how $[n]$ is colored,
proving Theorem \ref{thm_lower_bound}.

\section{Conclusion and New Directions}

We have shown that all $3$-term equations are uncommon over $[n]$.
For any single equation $a_1 x_1 + \cdots + a_k x_k = 0$ over
an abelian group $A$ whose order is relatively prime to each $a_i$,
a full classification is known.
\begin{theorem}[\cite{V}]\label{thm_groups}
An equation is uncommon over $A$ if and only if $k$ is even 
and has no canceling partition.
\end{theorem}
A \textbf{canceling partition} of an equation is 
a partition of the coefficients into pairs
$\{a_i, a_j\}$ such that $a_i + a_j = 0$.
Over $[n]$, we expect the following to be true.
\begin{conjecture}\label{conj_main}
An equation is common over $[n]$ if and only if $k$ is even
and has a canceling partition.
\end{conjecture}
Note that if this were true,
equations with $k$ even would behave the same over $A$ and $[n]$,
while equations with $k$ odd would behave differently.
With this paper, the conjecture is now confirmed for $k = 3$.
Much is still unknown, but we can also definitively say that 
equations with $k$ even and no canceling partition are uncommon over $[n]$.
This is simply because they are uncommon over $\Z_p$ if $p$
is a large enough prime ($p > \max\{|a_i|\}$) by Theorem \ref{thm_groups},
and Lemma \ref{lem_mod_ok} implies they are also uncommon over $[n]$.
Appendix \ref{xyzw-app} provides a proof for the only type of equation
known to be common over $[n]$:
\begin{equation}
x_1 + \cdots + x_{k/2} = x_{k/2+1} + \cdots + x_{k}
\;\; (k \text{ even}).
\end{equation}
For instance, it is not even known if $x + 2y = z + 2w$
is common.

Aside from these types of classification problems
(and ones which address systems of equations as in \cite{KLN}), 
improving upper and lower bounds on minima remains widely open.
Furthermore, all these questions and more can be asked about
colorings of more than 2 colors.

\section*{Acknowledgments}

We would like to thank Zhanar Berikkyzy 
for many enlightening discussions and helpful 
suggestions throughout our work on this paper.
We would also like to thank an anonymous reviewer
whose suggestion to use Fourier-analytic techniques
greatly streamlined some of our arguments
and was instrumental in the writing of 
Section \ref{subsec_fourier}.

\bibliographystyle{amsplain}
\bibliography{3-term_eqs}

\appendix

\section{Monochromatic Solutions
for Additive Tuples}\label{xyzw-app}

Fix $k \in \N$ even.
Here we prove all \textit{additive tuples}, equations of the form
\begin{equation}\label{eq_additive_tuple}
x_1 + \cdots + x_{k/2} = x_{k/2 + 1} + \cdots + x_k,
\end{equation}
are common over $[n]$,
i.e. the minimum fraction of monochromatic solutions
the same as what is expected from uniformly random colorings: $2^{1-k}$.
Let $p$ be a prime which is larger than $kn/2$.  
We identify $[n]$ with the subset $S = \{1, 2, \dots, n\} \subseteq \Z_p$
and note that by our choice of $p$ any solution to (\ref{eq_additive_tuple})
over $S$ is also a solution over the integers.  

Let $\mathds{1}_S$ be the indicator function of $S$, and 
recall the definition of the Fourier transform from 
Section \ref{subsec_fourier}: 
$$\widehat{ \mathds{1}_S}(\xi) 
= \frac{1}{p}\sum_{t \in \Z_p} \mathds{1}_S(t) e^{-2 \pi i \xi t/p}. $$

It is standard (see, for example, Equation (4.14) in \cite{TV}) that the number of solutions to (\ref{eq_additive_tuple}) in $S$ is given by 
$$p^{k-1}\sum_{t \in \Z_p}  \left| \widehat{ \mathds{1}_S}(t) \right|^{k}.$$

Now suppose we have a partition of $S$ into a red set $R$ and a blue set $B$.  
The total number of monochromatic solutions to 
$x_1 + \cdots + x_{k/2} = x_{k/2+1} + \cdots + x_{k}$ is then given by 
$$p^{k-1}\sum_{j \in \Z_p} \left| \widehat{ \mathds{1}_R}(j) \right|^{k} 
+ p^{k-1}\sum_{j \in \Z_p} \left| \widehat{ \mathds{1}_B}(j) \right|^{k}. $$
Using the inequality $x^{k} + y^{k} \geq 2^{1 - k} (x+y)^{k}$, 
which is valid for all real $x, y$ and $k \geq 2$ even
(by Jensen's inequality), we get that the number of monochromatic solutions is at least 

\begin{eqnarray*}
(p/2)^{k-1} \sum_{j \in \Z_p} 
\left( |\widehat{\mathds{1}_{R}}(j)| + | \widehat{\mathds{1}_{B}}(j)|\right)^k 
&\geq& 
(p/2)^{k-1} \sum_{j \in \Z_p} 
\left|  \widehat{\mathds{1}_{R}}(j) + \widehat{\mathds{1}_{B}}(j) \right|^k \\
&=& 
2^{1 - k} \left(p^{k-1}\sum_{j \in \Z_p} \left| \widehat{ \mathds{1}_S}(j) \right|^k\right).
\end{eqnarray*}
In other words, the number of monochromatic solutions is always at least a $2^{1 - k}$ fraction of the total number of solutions.  

\section{Computations}\label{app_computations}

Below are detailed calculations for the equations remaining
after using the Fourier-analytic techniques, 
Proposition \ref{prop_high_gcd},
and past results \cite{RZ, BCG, TW}.

\subsection{$2x - 2y + cz = 0$}

Recall the colorings $f : [n] \to \{-1, 1\}$ used for these equations:
\[f(t) = \begin{cases}
-1, & t \text{ even, } t \leq \alpha n, \\
1, & \text{otherwise},
\end{cases}
\qquad \text{where} \qquad
\alpha = \begin{cases}
3/4, & c = 1, \\
2/c, & c \geq 3
\end{cases}\]
(note that $c$ is odd because our equations
are fully reduced).
We address the case when $c \geq 3$ and $c = 1$ separately.

Let $c \geq 3$. Since in any solution $z$ is even,
this coloring
forces $z$ to be blue:
if $z \in [2n/c,n]$,
then 
$$2(y - x) = cz \geq c(2n/c) = 2n,$$
so
$y - x \geq n$,
but this is not possible.
Therefore, all monochromatic 
solutions are blue, and in particular
$x,y,z \in [1,2n/c]$.
There are $2n^2/c^2 + O(n)$ ways to
choose two numbers $x$ and $y$ in $[1, 2n/c]$
(note $y > x$ is required for a valid 
solution to the equation).
Only $1/4 + O(n^{-1})$ of the pairs $(x,y)$ 
are blue\footnote{The $O(n^{-1})$ error term 
here is due to edge effects from the boundaries 
of the regions; 
the total number of pairs involved in such effects 
is $O(n)$.
We use facts similar to this several more times
throughout this Appendix.}.
Furthermore, $2(y - x)$ must be divisible by $c$,
and only $1/c + O(n^{-1})$ of the pairs $(x,y)$ 
meet that requirement.
Finally, once $x$ and $y$ are chosen,
$z$ is determined,
and $z$ is always in $[1,2n/c]$:
$z = 2(y - x)/c < 2n/c$.
Therefore,
there are
$$\frac{n^2}{2c^3} + O(n)$$
monochromatic solutions.
The total number of solutions is $n^2/2c + O(n)$:
there are $\binom{n}{2}$ ways to
choose two numbers in $[n]$ and set them as $x$ and $y$,
and $1/c + O(n^{-1})$ of these pairs will have
$z = 2(y - x) / c \in \Z$
(and $z$ will always be in $[n]$).
This gives us
$$\mu_{\{2x - 2y + cz = 0\}}(f) 
\leq \frac{n^2/2c^3 + O(n)}{n^2/2c + O(n)}
	= \frac{1}{c^2} + o(1)
	= \frac{1}{4} - \Omega(1)
\quad \text{ (for $c \geq 3$).}$$

Now let $c = 1$. We will use the fact that for any solution
$z$ must be even
and break the counting into two cases:
(a) $z \in [1,3n/4]$ (blue solutions) and 
(b) $z \in [3n/4,n]$ (red solutions).
To count the number of monochromatic solutions,
it helps to visualize solutions on
an $n \times n$ grid.
For our purposes here
the horizontal axis will 
represent the $x$ values,
and the vertical axis will represent 
the $y$ values.
Once $x$ and $y$ are chosen,
$z = 2(y - x)$ is determined,
and valid solutions $(x, y, 2(y - x))$ in $[n]^3$
will lie within a certain area on the grid.
Figure \ref{area-calcs-fig} is provided as a visual aid 
for the following computations.

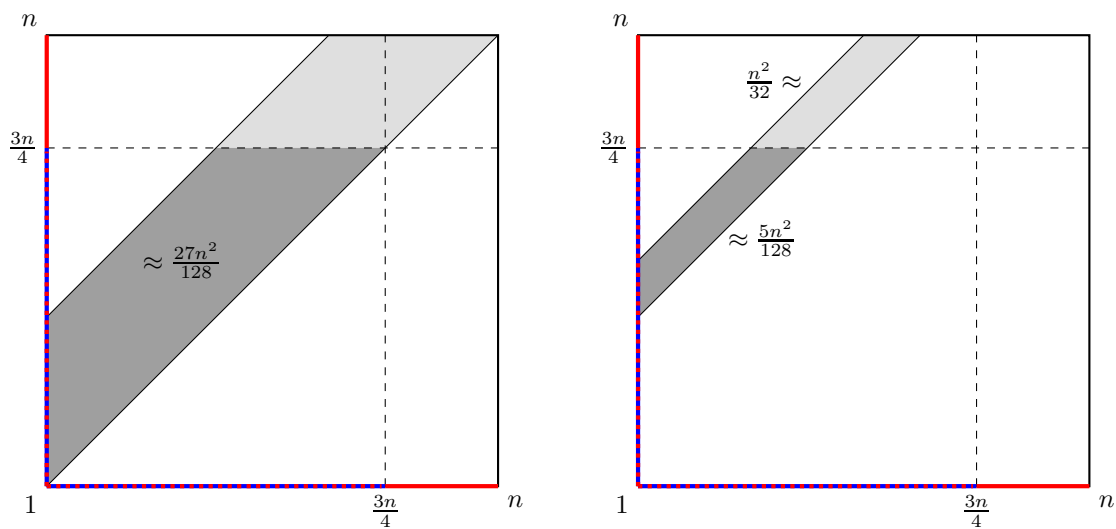
\begin{figure}[hbt!]
\centering
\begin{subfigure}{0.48\textwidth}
\centering
\begin{tikzpicture}[scale=0.6]
\fill[gray!25] (0,0) -- (0,3.75) --
	(6.25,10) -- (10,10) -- cycle;
	
\fill[gray!75] (0,0) -- (0,3.75) --
	(3.75,7.5) -- (7.5,7.5) -- cycle;

\draw[thick] (0,0) -- (10,0) -- (10,10) 
	-- (0,10) -- (0,0);

\draw[dashed] (7.5,-0) -- (7.5,10);
\draw[dashed] (-0,7.5) -- (10,7.5);

\draw (0,0) -- (10,10);
\draw (0,3.75) -- (6.25,10);

\draw[blue, ultra thick] (0,0) -- (7.5,0);
\draw[red, ultra thick, dotted] (0,0) -- (7.42,0);
\draw[blue, ultra thick] (0,0) -- (0,7.5);
\draw[red, ultra thick, dotted] (0,0) -- (0,7.42);
\draw[red, ultra thick] (7.5,0) -- (10,0);
\draw[red, ultra thick] (0,7.5) -- (0,10);

\draw (0,0) node[anchor=north east]{1};
\draw (10,0) node[anchor=north west]{$n$};
\draw (0,10) node[anchor=south east]{$n$};
\draw (0,7.5) node[anchor=east]{$\frac{3n}{4}$};
\draw (7.5,0) node[anchor=north]{$\frac{3n}{4}$};

\draw (3,5) node{$\approx \frac{27n^2}{128}$};
\end{tikzpicture}
\caption{$z \in [1,3n/4]$:
The gray areas combined represent all  
pairs $(x,y)$ with
$1 \leq z = 2(y - x) \leq 3n/4$.
Since here we are counting blue solutions,
$x,y \in [1,3n/4]$, as well,
i.e. we only consider the dark gray area.
The area of the dark gray trapezoid must
be multiplied by $1/4$, since only about $1/4$
of the pairs $(x,y)$ in that region are blue.\\}
\end{subfigure}
\hfill
\begin{subfigure}{0.48\textwidth}
\centering
\begin{tikzpicture}[scale=0.6]
\fill[gray!25] (0,5) -- (0,3.75) --
	(6.25,10) -- (5,10) -- cycle;
	
\fill[gray!75] (0,5) -- (0,3.75) --
	(3.75,7.5) -- (2.5,7.5) -- cycle;

\draw[thick] (0,0) -- (10,0) -- (10,10) 
	-- (0,10) -- (0,0);

\draw[dashed] (7.5,-0) -- (7.5,10);
\draw[dashed] (-0,7.5) -- (10,7.5);

\draw (0,5) -- (5,10);
\draw (0,3.75) -- (6.25,10);

\draw[blue, ultra thick] (0,0) -- (7.5,0);
\draw[red, ultra thick, dotted] (0,0) -- (7.42,0);
\draw[blue, ultra thick] (0,0) -- (0,7.5);
\draw[red, ultra thick, dotted] (0,0) -- (0,7.42);
\draw[red, ultra thick] (7.5,0) -- (10,0);
\draw[red, ultra thick] (0,7.5) -- (0,10);

\draw (0,0) node[anchor=north east]{1};
\draw (10,0) node[anchor=north west]{$n$};
\draw (0,10) node[anchor=south east]{$n$};
\draw (0,7.5) node[anchor=east]{$\frac{3n}{4}$};
\draw (7.5,0) node[anchor=north]{$\frac{3n}{4}$};

\draw (2.75,5.5) node{$\approx \frac{5n^2}{128}$};
\draw (3,9) node{$\frac{n^2}{32} \approx$};
\end{tikzpicture}
\caption{$z \in [3n/4, n]$:
The gray areas combined represent all 
pairs $(x,y)$ 
with $3n/4 \leq z = 2(y - x) \leq n$.
Since we are counting red solutions,
the dark gray area must be multiplied by $1/4$,
because only about $1/4$ of the pairs $(x,y)$ in that
trapezoid are red,
and the light gray area must be multiplied by
$1/2$, because only about half of the $x$ values there are
red (the $y$ values in the light gray area
are all red).}
\end{subfigure}
\caption{Two depictions of the $n \times n$
grid in the $xy$-plane.}
\label{area-calcs-fig}
\end{figure}

(a) For a blue solution, we must have $x,y,z \in [1,3n/4]$.
There are $27n^2/128 + O(n)$ valid 
choices for $x$ and $y$ in $[1,3n/4]$
that also lead to $z \in [1,3n/4]$.
Note, however, that only $1/4 + O(n^{-1})$ of the pairs
$(x,y)$ will be blue.
Therefore, there are 
$$\frac{27}{512}n^2 + O(n)$$
blue solutions.

(b) For a red solution, note that since $z$ must
be even, $z \in [3n/4,n]$.
For valid $x$ and $y$, there are two 
possible cases here: (i) $x \in [1,3n/4]$
and $y \in [3n/4,n]$, or (ii) $x,y \in [1,3n/4]$.
In (i) there are $n^2/32 + O(n)$
valid choices for $x$ and $y$, 
but only $1/2 + O(n^{-1})$ of the $x$ will be red.
Therefore, the contribution from (i) is
$$\frac{1}{64}n^2 + O(n).$$
In (ii) there are $5n^2/128 + O(n)$
solutions, but only $1/4 + O(n^{-1})$ of the $(x,y)$ 
will be red,
so the contribution from (ii) is
$$\frac{5}{512}n^2 + O(n).$$
Adding up all the blue solutions 
and all the red solutions, we get
$$\left(\frac{27}{512} + \frac{1}{64} + \frac{5}{512}
\right)n^2 + O(n) = \frac{5}{64}n^2 + O(n)$$
monochromatic solutions.

The total number of solutions is $3n^2/8 + O(n)$,
because for a solution we must have
$$1 \leq z = 2(y - x) \leq n,$$
or $0 < y - x \leq n/2$,
and there are $3n^2/8 + O(n)$ pairs $(x,y)$
which satisfy this.
Therefore, 
$$\mu_{\{2x - 2y + z = 0\}}([n]) \leq \frac{5}{24} + o(1) = \frac{1}{4} - \Omega(1),$$
i.e. $2x - 2y + z = 0$ is uncommon over $[n]$.

\subsection{$2x - y + 2z = 0$}

We will now cover a general technique 
to show an individual equation is uncommon,
and then we will use it on the equation
$2x - y + 2z = 0$.
Fix the equation $ax + by + cz = 0$,
and let $f : [n] \to \{-1,1\}$ be a coloring.
Consider the value
\begin{equation}\label{orig-counting}
L = \sum_{ai + bj + ck = 0}
f(i) f(j) + f(i) f(k) + f(j) f(k).
\end{equation}
Here and elsewhere in this section, 
the variables $i,j,k$ are implicitly
assumed to lie in $[n]$.
$L$, in a sense, indirectly counts the number
of monochromatic solutions:
by direct computation,
each summand is $3$ if $i,j,k$ are monochromatic
and is $-1$ otherwise,
so
$$L = 3(\text{\# monochr. solutions}) 
	- (\text{\# non-monochr. solutions}).$$
With a straightforward manipulation, we get
\begin{equation}\label{N-negative-eq}
\text{\# monochr. solutions}
	= \frac{1}{4}(\text{\# total solutions})
	+ \frac{L}{4},
\end{equation}
which means that to show $ax + by + cz = 0$
is uncommon,
we only need to exhibit a family of colorings
with $L = C n^2 + O(n)$
for some $C < 0$.

Our next task is to find a way to actually compute $L$.
For $i < j$, let $N(i,j)$ denote the number of times
a solution contains $i$ and $j$ as two of the three
values for $x,y,z$ (in no particular order).
Then we can rewrite (\ref{orig-counting}) as
\begin{equation}\label{v2-counting}
L = \sum_{i < j} N(i,j)f(i) f(j) + O(n).
\end{equation}
Note the $O(n)$ term accounts for 
the possibility of solutions with $i = j$.
We can view $N(i,j)$ as the sum of 
six indicator-like functions, 
each corresponding to where there exists 
a solution with $(i,j)$ playing the role 
of some ordered pair from $\{x,y,z\}$.
We will examine the total contribution of each of these 
functions separately in computing $L$,
using areas in an $n \times n$ grid to aid the calculations.

The coloring
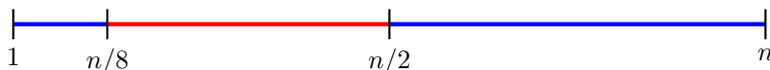
\begin{center}
\begin{tikzpicture}
\draw[ultra thick, blue] (0,0) -- (1.25,0);
\draw[ultra thick, red] (1.25,0) -- (5, 0);
\draw[ultra thick, blue] (5,0) -- (10,0);
\draw[thick] (0,0.2) -- (0,-0.2);
\draw (0,-0.2) node[anchor = north]{1};
\draw[thick] (1.25,0.2) -- (1.25,-0.2);
\draw (1.25,-0.2) node[anchor = north]{$n/8$};
\draw[thick] (5,0.2) -- (5,-0.2);
\draw (5,-0.2) node[anchor = north]{$n/2$};
\draw[thick] (10,0.2) -- (10,-0.2);
\draw (10,-0.2) node[anchor = north]{$n$};
\end{tikzpicture}
\end{center}
will be enough for our purposes\footnote{
This coloring was obtained by first running a basic version
of the local optimization algorithm 
described in \cite{BCG} for $n = 1000$.
We then simplified the coloring by hand
and blew it up to
an arbitrary $n$.
The hand-manipulation did increase the
number of monochromatic solutions slightly,
but it greatly simplified the following calculations.}.
To compute $L$,
the cases to consider are
\begin{enumerate}
\item $(i,j)$ plays the role of $(x,z)$: $2i - y + 2j = 0$;
restriction: $1 \leq 2i + 2j \leq n$.
\item $(i,j)$ plays the role of $(z,x)$: $2j - y + 2i = 0$;
restriction: $1 \leq 2j + 2i \leq n$.
\item $(i,j)$ plays the role of $(x,y)$: $2i - j + 2z = 0$;
restrictions: $2 \leq j - 2i \leq 2n$, $j$ even.
\item $(i,j)$ plays the role of $(y,x)$: $2j - i + 2z = 0$;
restrictions: $2 \leq i - 2j \leq 2n$, $i$ even.
\item $(i,j)$ plays the role of $(y,z)$: $2x - i + 2j = 0$;
restrictions: $2 \leq i - 2j \leq 2n$, $i$ even.
\item $(i,j)$ plays the role of $(z,y)$: $2x - j + 2i = 0$;
restrictions: $2 \leq j - 2i \leq 2n$, $j$ even.
\end{enumerate}
Note in each of these six cases one of the bounds
holds trivially.

Let us explore Case 1.
We start by defining an ``indicator'' of sorts,
which will help us rewrite (\ref{v2-counting}):
$$I_1(i,j) = 
\begin{cases}
f(i)f(j), & 1 \leq 2i + 2j \leq n, \\
0, & \text{otherwise}.
\end{cases}
$$
Aggregating, we define $L_1 = \sum_{i < j}I_1(i,j)$,
which simply counts up Case 1's 
contribution to (\ref{v2-counting}).
Note Case 2 is identical to Case 1.

We can approach Case 3 in a similar manner,
but there is an additional twist.
If we define
$$I_3(i,j) = 
\begin{cases}
f(i)f(j), & 2 \leq j - 2i \leq 2n, \\
0, & \text{otherwise},
\end{cases}
$$
then $L_3 = \sum_{i<j}I_3(i,j)$
includes the contributions from 
both even and odd $j$,
so the contribution from Case 3 is
actually $\frac{1}{2}L_3 + O(n)$.
The rest of the $L_r$ are defined similarly.

Each $L_r$ can be computed by considering the 
pairs $(i,j)$ in Case $r$ with $i < j$
and subtracting 
the number of dichromatic pairs from the number
of monochromatic pairs.
Similar to (yet distinct from)
the counting technique 
implemented for the equation $2x - 2y + z = 0$,
to compute a given $L_r$
we can consider areas within an $n \times n$
grid, as seen in Figure \ref{221-visual},
now with $i$ represented on the horizontal axis
and $j$ on the vertical axis.
\begin{figure}[hbt!]
\centering
\begin{subfigure}{0.48\textwidth}
\centering
\begin{tikzpicture}[scale=0.64]
\fill[gray!50] (0,0) -- (0,5) -- (5,0) -- cycle;

\fill[blue] (0,0) -- (0,1.25)
 -- (1.25,1.25) -- (1.25,0) -- cycle;

\fill[red] (1.25,1.25) -- (1.25,3.75)
	-- (3.75,1.25) -- cycle;

\draw[thick] (0,0) -- (10,0) -- (10,10) 
	-- (0,10) -- (0,0);
	
\draw (0,5) -- (5,0);

\draw[dashed] (1.25,-0) -- (1.25,10);
\draw[dashed] (5,-0) -- (5,10);
\draw[dashed] (-0,1.25) -- (10,1.25);
\draw[dashed] (-0,5) -- (10,5);

\draw (1.25,0) node[anchor=north]{$\frac{n}{8}$};
\draw (5,0) node[anchor=north]{$\frac{n}{2}$};
\draw (0,1.25) node[anchor=east]{$\frac{n}{8}$};
\draw (0,5) node[anchor=east]{$\frac{n}{2}$};
\draw (0,0) node[anchor=north east]{1};
\draw (10,0) node[anchor=north]{$n$};
\draw (0,10) node[anchor=east]{$n$};

\fill[white, opacity=0.9] (0,0) -- (10,10) 
	-- (10,0) -- cycle;
\draw[gray] (0,0) -- (10,10);

\draw[blue, ultra thick] (0,0) -- (1.25,0);
\draw[blue, ultra thick] (0,0) -- (0,1.25);
\draw[red, ultra thick] (1.25,0) -- (5,0); 
\draw[red, ultra thick] (0,1.25) -- (0,5); 
\draw[blue, ultra thick] (5,0) -- (10,0);
\draw[blue, ultra thick] (0,5) -- (0,10);
\end{tikzpicture}
\caption{Visual for $L_1$}
\end{subfigure}
\hfill
\begin{subfigure}{0.48\textwidth}
\centering
\begin{tikzpicture}[scale=0.64]

\fill[gray!50] (0,0) -- (0,10) -- (5,10) -- cycle;

 
\fill[blue] (0,0) -- (0.625,1.25)
 -- (0,1.25) -- cycle;
 

\fill[red] (1.25, 2.5) -- (2.5,5) -- (1.25,5) -- cycle;

	
\fill[blue] (0,5) -- (1.25,5) -- (1.25,10)
	-- (0,10) -- cycle;

\draw[thick] (0,0) -- (10,0) -- (10,10) 
	-- (0,10) -- (0,0);
	
\draw (0,0) -- (5,10);

\draw[dashed] (1.25,-0) -- (1.25,10);
\draw[dashed] (5,-0) -- (5,10);
\draw[dashed] (-0,1.25) -- (10,1.25);
\draw[dashed] (-0,5) -- (10,5);

\fill[white, opacity=0.9] (0,0) -- (10,10) 
	-- (10,0) -- cycle;
\draw[gray] (0,0) -- (10,10);
	
\draw (1.25,0) node[anchor=north]{$\frac{n}{8}$};
\draw (5,0) node[anchor=north]{$\frac{n}{2}$};
\draw (0,1.25) node[anchor=east]{$\frac{n}{8}$};
\draw (0,5) node[anchor=east]{$\frac{n}{2}$};
\draw (0,0) node[anchor=north east]{1};
\draw (10,0) node[anchor=north]{$n$};
\draw (0,10) node[anchor=east]{$n$};

\draw[blue, ultra thick] (0,0) -- (1.25,0);
\draw[blue, ultra thick] (0,0) -- (0,1.25);
\draw[red, ultra thick] (1.25,0) -- (5,0); 
\draw[red, ultra thick] (0,1.25) -- (0,5); 
\draw[blue, ultra thick] (5,0) -- (10,0);
\draw[blue, ultra thick] (0,5) -- (0,10);
\end{tikzpicture}
\caption{Visual for $L_3$}
\end{subfigure}
\caption{In each case, the contributions
to $L$ are computed by subtracting the 
gray area (dichromatic pairs)
from the red/blue area (monochromatic pairs).
The lighter regions represent $i \geq j$
and are not a part of $L$.}
\label{221-visual}
\end{figure}
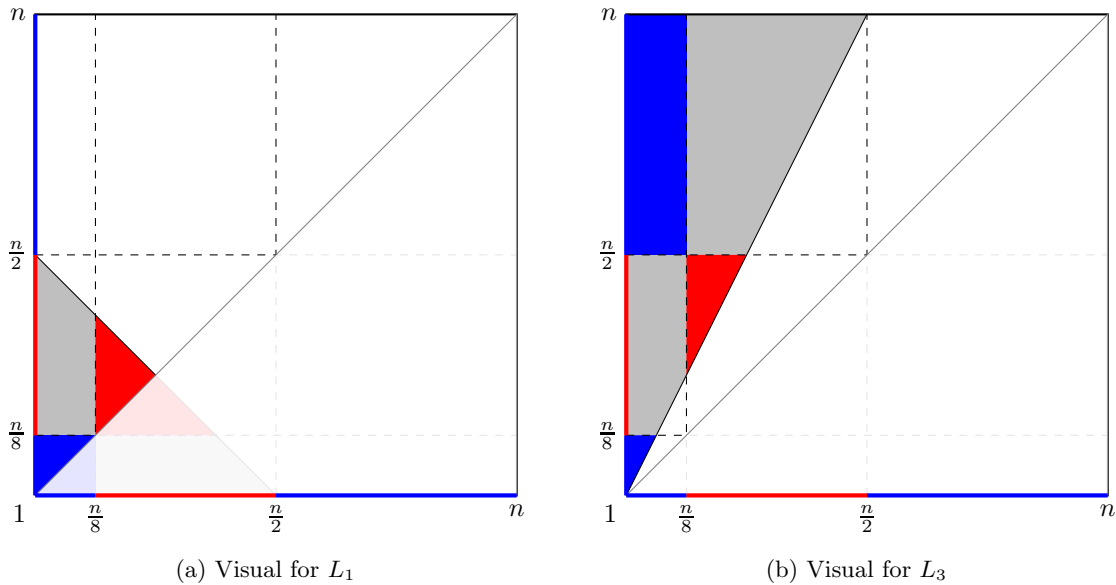

Case 2 is identical to 1, Case 5 
is identical to 3, and Cases 4 and 6
lie completely in $i \geq j$ and therefore will
not contribute to $L$.
This allows us to simplify the calculation:
\begin{equation}\label{L-221-eq}
L = L_1 + L_2 + \frac{1}{2}(L_3 + L_5) + O(n)
= 2 L_1 + L_3 + O(n) = -\frac{15}{128}n^2 + O(n).
\end{equation}
The coefficient of $n^2$ is negative, so
by (\ref{N-negative-eq})
this coloring gives (asymptotically)
fewer monochromatic solutions
than what is expected from uniformly 
random colorings, i.e. $2x - y + 2z = 0$ is uncommon over 
$[n]$.

\end{document}